\documentclass[12pt]{article}

\usepackage{titling}
\usepackage[numbers,square,comma]{natbib}
\usepackage{float}
\usepackage[pdfpagelabels]{hyperref}
\usepackage{amsmath}
\usepackage{amssymb}
\usepackage{amsthm}
\usepackage{graphicx}
\usepackage{geometry}
\usepackage{tikz}
\usepackage{tkz-graph}
\usepackage{chngcntr}
\usepackage{thmtools}
\usepackage{thm-restate}
\usepackage{authblk}
\usepackage{tabularx}
\usepackage{environ}
\usepackage{bbold}
\usetikzlibrary{shapes}

\hypersetup{%
    pdftitle={Distance-Restricted Firefighting on Finite Graphs},
    pdfauthor={John Marcoux, Andrea C. Burgess, David A. Pike, John Hawkin, Alexander Howse},                                                   
    pdfkeywords={graph theory, combinatorics},             
    colorlinks = true,                                                          
    linkcolor = black,                                                          
    anchorcolor = black,                                                        
    citecolor = black,                                                          
    filecolor = black,                                                          
    urlcolor = black,                                                           
    pdfprintscaling=None                                                        
 }

\counterwithin{figure}{section}

\makeatletter
\newcommand{\problemtitle}[1]{\gdef\@problemtitle{#1}}
\newcommand{\probleminput}[1]{\gdef\@probleminput{#1}}
\newcommand{\problemquestion}[1]{\gdef\@problemquestion{#1}}
\NewEnviron{problem}{
  \problemtitle{}\probleminput{}\problemquestion{}
  \BODY
  \par\addvspace{.5\baselineskip}
  \noindent
  \begin{tabularx}{\textwidth}{@{\hspace{\parindent}} l X c}
    \multicolumn{2}{@{\hspace{\parindent}}l}{\@problemtitle} \\
    \textbf{Input:} & \@probleminput \\
    \textbf{Question:} & \@problemquestion
  \end{tabularx}
  \par\addvspace{.5\baselineskip}
}
\makeatother

\title{Distance-Restricted Firefighting on Finite Graphs}

\author[1]{Andrea C. Burgess}
\author[2]{John Hawkin}
\author[3]{Alexander Howse}
\author[4]{\\John Marcoux}
\author[5]{David A. Pike}
\affil[1]{Department of Mathematics and Statistics, University of New Brunswick, Saint John, NB, E2L 4L5, Canada. \texttt{andrea.burgess@unb.ca}}
\affil[2]{Nasdaq Verafin, 18 Hebron Way, St. John's, NL, A1A 0L9. \texttt{john.hawkin@nasdaq.com}}
\affil[3]{Nasdaq Verafin, 18 Hebron Way, St. John's, NL, A1A 0L9. \texttt{alexander.howse@nasdaq.com}}
\affil[4]{Department of Mathematics, Toronto Metropolitan University, Toronto, ON, M5B 2K3, Canada. \texttt{jmarcoux@torontomu.ca}}
\affil[5]{Department of Mathematics and Statistics, Memorial University of Newfoundland, St. John's, NL, A1C 5S7, Canada. \texttt{dapike@mun.ca}}
\date{\today}

\newtheorem{theorem}{Theorem}[section]

\newtheorem{lemma}[theorem]{Lemma}
\newtheorem{conjecture}[theorem]{Conjecture}

\theoremstyle{definition}

\begin{document}

\maketitle
\medskip 

\begin{abstract}
    In the classic version of the game of firefighter, on the first turn a fire breaks out on a vertex in a graph $G$ and then $b$ firefighters protect $b$ vertices. On each subsequent turn, the fire spreads to the collective unburned neighbourhood of all the burning vertices and the firefighters again protect $b$ vertices. Once a vertex has been burned or protected it remains that way for the rest of the game. In \textit{distance-restricted firefighting} the firefighters' movement is restricted so they can only move up to some fixed distance $d$ and they may or may not be permitted to move through burning vertices. In this paper we establish the NP-completeness of the distance-restricted versions of {\sc $b$-Firefighter} and present an integer program for computing the exact value. We also discuss some interesting properties of the \textit{Expected Damage} function.
\end{abstract}

\section{Introduction}\label{sec:intro}

The firefighter problem, introduced by Hartnell in 1995~\cite{hartnell1995firefighter}, is a turn-based game played on a graph where fire is a model for some contagion that spreads from a set of `burning' vertices to their collective neighbourhoods on each turn. For our purposes we consider the game to be a process which takes place over a series of discrete rounds. Initially there is some set of burning vertices. In each round each firefighter protects a single vertex from the fire, the consequence being that the fire does not spread to these vertices for the remainder of the game. On each subsequent turn, the fire spreads to the unprotected neighbourhoods of the burning vertices and the firefighters protect a new set of vertices. The notion of \textit{distance-restricted} firefighting was introduced in~\cite{burgess2022distance} as a new variant of the firefighter problem. In this version the firefighters were restricted to moving up to some distance $d$ each turn and could not move through any burning vertices. We make note of two things now. First, even though the firefighters may `leave' a vertex after defending it, it still remains defended for the remainder of the game. Second, in the first round the firefighters do not have a previous position to move from so we allow them to begin at any unburned vertex. Up to this point, distance-restricted firefighting has only been studied in the context of the containment problem on infinite graphs~\cite{burgess2022distance, chen2017continuous, days2019firefighter} whereas the present paper focuses on problems on finite graphs.

One of the key items of interest in firefighting is the \textit{Maximum Vertices Saved} ({\sc MVS}($G$,$F$;$b$)) function. Given a graph $G$, a set $F$ of initially burning vertices, and a number $b$ of firefighters, {\sc MVS}($G$,$F$;$b$) is equal to the maximum number of vertices the firefighters can save. The related decision problem asking if {\sc MVS}($G$,$F$;$b$) is at least some constant $k$ is referred to as {\sc $b$-Firefighter}. Another decision problem is {\sc ($S$,$b$)-Fire} where given a graph, a set of initially burning vertices and a set of vertices $S$ we ask if all the vertices in $S$ can be saved using only $b$ firefighters. Both {\sc $b$-Firefighter} and {\sc ($S$,$b$)-Fire} are known to be NP-complete even when restricted to certain classes of trees. These are both consequences of the main result from~\cite{bazgan2013firefighter} (which builds on a result from~\cite{finbow2007npcomplete}) where it is shown that the {\sc Max ($S$,$b$)-Fire} problem is NP-hard. Note that the value of $b$ is fixed as part of the problem specification, and is not itself an input parameter.

    \begin{problem}
        \problemtitle{{\sc Max ($S$,$b$)-Fire}}
        \probleminput{A graph $G$, a set of vertices $S \subseteq V(G)$, and a vertex $u \in V(G)$}
        \problemquestion{What is the maximum number of vertices in $S$ which can be saved using $b$ firefighters if a fire breaks out at $u$?}
    \end{problem}

    We recommend the survey by Finbow and MacGillivray~\cite{finbow2009firefighter} and the recent thesis by Wagner~\cite{wagner2021new} for further reading on these problems as well as other variants of the firefighter problem.

    For our purposes we define two variants of the {\sc MVS} function, \textit{Distance-Restricted Maximum Vertices Saved} ({\sc DRMVS}($G$,$F$;$b$,$d$)) and \textit{Distance-Path-Restricted Maximum Vertices Saved} ({\sc DPRMVS}($G$,$F$;$b$,$d$)). With {\sc DRMVS} each firefighter can move up to some fixed distance $d$ (which is now included in the problem parameters). For {\sc DPRMVS} the firefighters can move up to some fixed distance $d$ and in addition the firefighters cannot move through the fire. We will see in Section~\ref{sec:complexity} that although these games may seem similar they can behave in very different ways in certain contexts.

    A common approach for analyzing optimization problems is to restate them as decision problems and analyze their complexity. In order to do this we first consider the {\sc $b$-Firefighter} decision problem, studied in~\cite{bazgan2013firefighter}.

    \begin{problem}
        \problemtitle{{\sc $b$-Firefighter}}
        \probleminput{A graph $G$, a set of vertices $F \subseteq V(G)$, and a positive integer $k$}
        \problemquestion{Is {\sc MVS}($G$,$F$;$b$) $ \geq k$?}
    \end{problem}

    \noindent This is easily extended to the two decision problems which will be the focus of this paper. Let $d$ be a positive integer and define the following decision problems, for which $b$ and $d$ are part of the problem specification.

    \begin{problem}
        \problemtitle{{\sc $d$-DR-$b$-FF}}
        \probleminput{A graph $G$, a set of vertices $F \subseteq V(G)$, and a positive integer $k$}
        \problemquestion{Is {\sc DRMVS}($G$,$F$;$b$,$d$) $\geq k$?}
    \end{problem}
    
    \begin{problem}
        \problemtitle{{\sc $d$-DPR-$b$-FF}}
        \probleminput{A graph $G$, a set of vertices $F \subseteq V(G)$, and a positive integer $k$}
        \problemquestion{Is {\sc DPRMVS}($G$,$F$;$b$,$d$) $\geq k$?}
    \end{problem}

    It is also worth noting that if we allow the case where $d=0$, {\sc $d$-DPR-$b$-FF} and {\sc $d$-DR-$b$-FF} are both similar to the {\sc Firebreak} problem~\cite{barnetson2021firebreak} where the firefighters defend for only the first turn and then allow the fire to spread as much as possible. In this case the only difference is that for our problems the number of firefighters is part of the problem statement but for {\sc Firebreak} it is part of the problem input. Note as well that if $d$ is at least the diameter of the graph then {\sc $d$-DR-$b$-FF} is equivalent to {\sc $b$-Firefighter} but under the same conditions {\sc $d$-DPR-$b$-FF} is not equivalent to {\sc $b$-Firefighter} in general.

    In the remaining sections of the paper, we first discuss the complexity of {\sc $d$-DR-$b$-FF} and {\sc $d$-DPR-$b$-FF} by showing their NP-completeness on graphs in general. We also show that the complexity of {\sc $d$-DPR-$b$-FF} is greatly reduced on trees and pose questions about the complexity of solving {\sc $d$-DPR-$b$-FF} on different graph classes. Following this, an extension of a $0,1$-integer program for solving the original {\sc MVS} function is explored which is capable of computing {\sc DRMVS} and a second extension which computes {\sc DPRMVS} is considered as well. In the penultimate section of the paper, the concept of expected damage is explored for distance-restricted and distance-path-restricted firefighting and we show that the monotonicity of the expected damage function for the original game does not hold for the distance-restricted game. The final section poses some problems which we consider natural next steps for this variant of firefighting.

    For a reference on graph theory we recommend~\cite{west2000introduction} and for complexity theory we recommend~\cite{arora2009computational}.

\section{Complexity of {\sc $d$-DR-$b$-FF} and {\sc $d$-DPR-$b$-FF}}\label{sec:complexity}
   
One of the ever present questions for problems like {\sc $d$-DR-$b$-FF} and {\sc $d$-DPR-$b$-FF} is whether or not they can be solved efficiently. It is already known that {\sc $b$-Firefighter} is NP-complete even when restricted to trees of maximum degree $b+2$~\cite{bazgan2013firefighter}. For {\sc $d$-DR-$b$-FF} and {\sc $d$-DPR-$b$-FF} we must first examine a fundamental building block of the game, which is determining whether or not a move is valid. If we cannot efficiently decide whether or not a move is valid then it is unlikely that we can solve either of these decision problems efficiently. To deal with this problem of checking a move's validity, we first build an auxiliary bipartite graph where one partite set $A$ represents the vertices defended at time $t$ and the other partite set $B$ represents the vertices defended at time $t+1$. Note that if multiple firefighters are on the same vertex at either of these time steps then we include one copy of the vertex for each firefighter. We add the edge $uv$ if the vertex from the original graph represented by $u \in A$ is distance at most $d$ from the vertex represented by $v \in B$ in the original graph. Of course in the path restricted case we modify this edge condition so the distance is determined in the subgraph induced by the unburned vertices. It is clear that a valid move in the game exists if and only if this auxiliary graph has a perfect matching, which can be determined in polynomial time~\cite{edmonds1965paths}.

    Now we can move on to the main theorem of this section, Theorem~\ref{thm:npcomplete}

\begin{theorem}\label{thm:npcomplete}
    {\sc $d$-DR-$b$-FF} and {\sc $d$-DPR-$b$-FF} are NP-complete when $d \geq 2$.
\end{theorem}

    In order to show this we begin by noting that the problems are both in NP since we can use the sequence of defended vertices as an easily verifiable certificate.

\begin{lemma}
    {\sc $d$-DR-$b$-FF} and {\sc $d$-DPR-$b$-FF} are in NP.
\end{lemma}
 
We will also consider the notion of weighted instances of {\sc $d$-DPR-$b$-FF} and {\sc $d$-DR-$b$-FF} in order to incentivize the firefighters to defend certain vertices.

    \begin{problem}
        \problemtitle{{\sc Weighted-$d$-DPR-$b$-FF}}
        \probleminput{A graph $G$, a set of vertices $F \subseteq V(G)$, a weight function $w:V(G) \to \mathbb{R}$, and a positive integer $k$}
        \problemquestion{Can $b$ firefighters save a set of vertices whose weights sum to at least $k$ while moving at most distance $d$ and without passing through a burning vertex if a fire breaks out at all the vertices in $F$?}
    \end{problem}
    
    \begin{problem}
        \problemtitle{{\sc Weighted-$d$-DR-$b$-FF}}
        \probleminput{A graph $G$, a set of vertices $F \subseteq V(G)$, a weight function $w:V(G) \to \mathbb{R}$, and a positive integer $k$}
        \problemquestion{Can $b$ firefighters save a set of vertices whose weights sum to at least $k$ while moving at most distance $d$ if a fire breaks out at all the vertices in $F$?}
    \end{problem}

    While it is not the focus of this paper, the weighted variant of firefighter is of independent interest since it helps to unify the different firefighter decision problems. For example, we can immediately use this weighted version of the problem to phrase a distance-path-restricted version of {\sc $(S,b)$-Fire} as a {\sc Weighted-$d$-DPR-$b$-FF} problem since we can assign a weight of $n$ (the number of vertices) to the vertices in $S$ and a weight of $1$ to the rest of the vertices and ask if it is possible to save a weight of $n|S|$. We will later use a similar concept of assigning larger and larger weights to vertices to require certain sets of vertices to take priority over others.

    We now move on to a series of reductions which will constitute a proof of Theorem~\ref{thm:npcomplete}.

\begin{theorem} \label{thm:fixed_d_hard}
    {\sc Weighted-$d$-DPR-$1$-FF} and {\sc Weighted-$d$-DR-$1$-FF} are NP-complete for $d \geq 2$.
\end{theorem}

\begin{proof}
    We present a reduction from {\sc $3$-SAT}~\cite{karp1972reducibility} to {\sc Weighted-$d$-DR-$1$-FF}. The reduction for {\sc Weighted-$d$-DPR-$1$-FF} is identical since the firefighter will never be required to move through the fire.

    Given a $3$-CNF formula $\phi'$ with $n$ variables and $m'$ clauses we construct a new formula $\phi$ with $n$ variables and $m = n + m'$ clauses by adding a clause of the form $x_i \lor x_i \lor \lnot x_i$ for every variable $x_i$ and we will assume these are the first $n$ clauses of $\phi$. The formula $\phi$ is satisfied exactly when $\phi'$ is satisfied and is only polynomially larger. 

    Now from $\phi$ construct a graph $G$ as follows. First add three vertices labelled $I, v_1, v_2$ to the graph and add an edge from $I$ to $v_1$ and an edge from $I$ to $v_2$. For every pair of literals $x_i,\lnot x_i$ add two vertices and label them $x_i$ and $\lnot x_i$. These will be referred to as \emph{variable vertices}. Additionally we will add $d$ extra pairs labelled $x_{n+i},\lnot x_{n+i}$ for $1 \leq i \leq d$. These additional vertices will be treated as variable vertices but they will have no bearing on the variable assignment during our reduction. Add all four possible edges between $v_1,v_2$ and $x_1, \lnot x_1$ and for $1 \leq i \leq n+d-1$ add all four possible edges between $x_i,\lnot x_i$ and $x_{i+1},\lnot x_{i+1}$. For each pair $x_i, \lnot x_i$ add four internally disjoint paths with $i$ internal vertices, all of which have $I$ as one endpoint and two of which have $x_i$ as their other endpoint while the other two have $\lnot x_i$ as their second endpoint. Note that all the internal vertices are new vertices which were not already part of the vertex set and this will be true of any paths which we add. See Figure~\ref{fig:npcompleteness_1} for an example of this part of the construction. For each clause $C_j = y_1^j \lor y_2^j \lor y_3^j$ (where $y_1^j,y_2^j,y_3^j$ represent the three literals in the $j^{th}$ clause), add a $K_3$ and label the vertices $y_1^j, y_2^j, y_3^j$. We will refer to these vertices as \emph{clause vertices}. For $1 \leq j \leq m - 1$ add three vertex disjoint paths with $d-1$ vertices, each with one end adjacent to a distinct vertex in $C_j$ and the other end adjacent to all vertices in $C_{j+1}$. See Figure~\ref{fig:npcompleteness_2} for an example. If the literal $z \in \{x_i, \lnot x_i\}$ appears in the clause $C_j$ as $y_t^j$ then add two internally disjoint paths between the vertex corresponding to $z$ and the vertex $y_t^j$ with $n + d - i + j - 1$ internal vertices, and two internally disjoint paths between the vertex corresponding to $\lnot z$ and the vertex $y_t^j$ with $n + d - i + j$ internal vertices. See Figure~\ref{fig:npcompleteness_3} for an example. Additionally add all six possible edges between the vertices of $C_1$ and the vertices $x_{n+d}, \lnot x_{n+d}$. Finally, give $v_1$ a weight of $|V(G)|^3$, each variable vertex a weight of $|V(G)|^2$ (including the last $d$ of them), each clause vertex a weight of $|V(G)|$, and all remaining vertices a weight of $1$. Overall, this yields an instance of {\sc Weighted-$d$-DR-$1$-FF} with the graph $G$, $I$ as the initial burning vertex, and $|V(G)|^3 + (n + d)|V(G)|^2 + m|V(G)|$ as the total weight which must be saved.

\begin{figure}[H]
    \centering
\begin{tikzpicture}
    
    \draw (6,0) node {$\cdots$};
   
    \SetUpVertex[FillColor=white]
    
    \tikzset{VertexStyle/.append style={minimum size=22pt, inner sep=0.5pt}}
    
    \Vertex[x=-1,y=0,L={$I$},]{I}
    
    \Vertex[x=1,y=1,L={$v_1$},]{V1}
    \Vertex[x=1,y=-1,L={$v_2$},]{V2}
    
    \Vertex[x=2.5,y=1,L={$x_1$},]{X1}
    \Vertex[x=2.5,y=-1,L={$\lnot x_1$},]{NX1}

    \Vertex[x=4,y=1,L={$x_2$},]{X2}
    \Vertex[x=4,y=-1,L={$\lnot x_2$},]{NX2}
    
    \tikzset{VertexStyle/.append style={minimum size=12pt}}
    
    \Vertex[x=1,y=3.5,NoLabel=true,]{PX2A1}
    \Vertex[x=1,y=3,NoLabel=true,]{PX2B1}
    \Vertex[x=1,y=2.5,NoLabel=true,]{PX1A}
    \Vertex[x=1,y=2,NoLabel=true,]{PX1B}
    
    \Vertex[x=1,y=-3.5,NoLabel=true,]{PNX2A1}
    \Vertex[x=1,y=-3,NoLabel=true,]{PNX2B1}
    \Vertex[x=1,y=-2.5,NoLabel=true,]{PNX1A}
    \Vertex[x=1,y=-2,NoLabel=true,]{PNX1B}
    
    \Vertex[x=2.5,y=-3.5,NoLabel=true,]{PNX2A2}
    \Vertex[x=2.5,y=-3,NoLabel=true,]{PNX2B2}
    \Vertex[x=2.5,y=3.5,NoLabel=true,]{PX2A2}
    \Vertex[x=2.5,y=3,NoLabel=true,]{PX2B2}
    
    \Edge(I)(V1)
    \Edge(I)(V2)
    
    \Edge(I)(PX2A1)
    \Edge(I)(PX2B1)
    \Edge(I)(PNX2A1)
    \Edge(I)(PNX2B1)
    
    \Edge(PX2A2)(PX2A1)
    \Edge(PX2B2)(PX2B1)
    \Edge(PNX2A2)(PNX2A1)
    \Edge(PNX2B2)(PNX2B1)
    
    \Edge(PX2A2)(X2)
    \Edge(PX2B2)(X2)
    \Edge(PNX2A2)(NX2)
    \Edge(PNX2B2)(NX2)
    
    \Edge(I)(PX1A)
    \Edge(I)(PX1B)
    \Edge(I)(PNX1A)
    \Edge(I)(PNX1B)
    
    \Edge(X1)(PX1A)
    \Edge(X1)(PX1B)
    \Edge(NX1)(PNX1A)
    \Edge(NX1)(PNX1B)
    
    \Edge(X1)(V1)
    \Edge(X1)(V2)
    \Edge(NX1)(V1)
    \Edge(NX1)(V2)
    
    \Edge(X1)(X2)
    \Edge(X1)(NX2)
    \Edge(NX1)(X2)
    \Edge(NX1)(NX2)
   
\end{tikzpicture}
    \caption{An example of the initial segment of the variable gadget.}
    \label{fig:npcompleteness_1}
\end{figure}

\begin{figure}[H]
    \centering
\begin{tikzpicture}
    \SetUpVertex[FillColor=white]
    
    \tikzset{VertexStyle/.append style={minimum size=12pt, inner sep=0.5pt}}
    
    \Vertex[x=0,y=1,L={$y_1^1$},]{Y11}
    \Vertex[x=-1,y=0,L={$y_2^1$},]{Y21}
    \Vertex[x=0,y=-1,L={$y_3^1$},]{Y31}
    
    \Vertex[x=1,y=1,NoLabel=true]{Z11}
    \Vertex[x=1,y=0,NoLabel=true]{Z21}
    \Vertex[x=1,y=-1,NoLabel=true]{Z31}
    
    \Vertex[x=5,y=1,L={$y_1^2$},]{Y12}
    \Vertex[x=4,y=0,L={$y_2^2$},]{Y22}
    \Vertex[x=5,y=-1,L={$y_3^2$},]{Y32}
     
    \Edge(Y11)(Y21)
    \Edge(Y31)(Y21)
    \Edge(Y11)(Y31)
   
    \Edge(Y11)(Z11)
    \Edge(Y21)(Z21)
    \Edge(Y31)(Z31)

    \Edge(Z11)(Y12)
    \Edge(Z11)(Y22)
    \Edge(Z11)(Y32)

    \Edge(Z21)(Y12)
    \Edge(Z21)(Y22)
    \Edge(Z21)(Y32)

    \Edge(Z31)(Y12)
    \Edge(Z31)(Y22)
    \Edge(Z31)(Y32)

    \Edge(Y12)(Y22)
    \Edge(Y32)(Y22)
    \Edge(Y12)(Y32)
   
\end{tikzpicture}
    \caption{An example of the clause gadgets and how they connect with $d = 2$.}
    \label{fig:npcompleteness_2}
\end{figure}

\begin{figure}[H]
    \centering
\begin{tikzpicture}
    \SetUpVertex[FillColor=white]
    
    \tikzset{VertexStyle/.append style={minimum size=22pt, inner sep=0.5pt}}
    
    \Vertex[x=10,y=2,L={$\lnot x_1$},]{X1}
    \Vertex[x=10,y=-1,L={$x_1$},]{NX1}
    
    \Vertex[x=0,y=2,L={$y_1^2$},]{Y12}
    \Vertex[x=-2,y=0,L={$y_2^2$},]{Y22}
    \Vertex[x=0,y=-1,L={$y_3^2$},]{Y32}
    
    \Edge(Y12)(Y22)
    \Edge(Y32)(Y22)
    \Edge(Y12)(Y32)
    
    \tikzset{VertexStyle/.append style={minimum size=12pt, inner sep=0.5pt}}
    
    \Vertex[x=1.5,y=2,NoLabel=true]{P1a}
    \Vertex[x=3,y=2,NoLabel=true]{P1b}
    \Vertex[x=4.5,y=2,NoLabel=true]{P1c}
    \Vertex[x=6,y=2,NoLabel=true]{P1d}
    \Vertex[x=7.5,y=2,NoLabel=true]{P1e}
    
    \Edge(Y12)(P1a)
    \Edge(P1a)(P1b)
    \Edge(P1b)(P1c)
    \Edge(P1c)(P1d)
    \Edge(P1d)(P1e)
    \Edge(P1e)(X1)
    
    \Vertex[x=1.5,y=1,NoLabel=true]{P2a}
    \Vertex[x=3,y=1,NoLabel=true]{P2b}
    \Vertex[x=4.5,y=1,NoLabel=true]{P2c}
    \Vertex[x=6,y=1,NoLabel=true]{P2d}
    \Vertex[x=7.5,y=1,NoLabel=true]{P2e}
  
    \Edge(Y12)(P2a)
    \Edge(P2a)(P2b)
    \Edge(P2b)(P2c)
    \Edge(P2c)(P2d)
    \Edge(P2d)(P2e)
    \Edge(P2e)(X1)
    
    \Vertex[x=1.5,y=0,NoLabel=true]{P3a}
    \Vertex[x=3,y=0,NoLabel=true]{P3b}
    \Vertex[x=4.5,y=0,NoLabel=true]{P3c}
    \Vertex[x=6,y=0,NoLabel=true]{P3d}
    \Vertex[x=7.5,y=0,NoLabel=true]{P3e}
    \Vertex[x=9,y=0,NoLabel=true]{P3f}
  
    \Edge(Y12)(P3a)
    \Edge(P3a)(P3b)
    \Edge(P3b)(P3c)
    \Edge(P3c)(P3d)
    \Edge(P3d)(P3e)
    \Edge(P3e)(P3f)
    \Edge(P3f)(NX1)
    
    \Vertex[x=1.5,y=-1,NoLabel=true]{P4a}
    \Vertex[x=3,y=-1,NoLabel=true]{P4b}
    \Vertex[x=4.5,y=-1,NoLabel=true]{P4c}
    \Vertex[x=6,y=-1,NoLabel=true]{P4d}
    \Vertex[x=7.5,y=-1,NoLabel=true]{P4e}
    \Vertex[x=9,y=-1,NoLabel=true]{P4f}
  
    \Edge(Y12)(P4a)
    \Edge(P4a)(P4b)
    \Edge(P4b)(P4c)
    \Edge(P4c)(P4d)
    \Edge(P4d)(P4e)
    \Edge(P4e)(P4f)
    \Edge(P4f)(NX1)

\end{tikzpicture}
\caption{An example of the connection between the first variable and the second clause with $d = 2$ when the formula has $3$ variables. We assume here that $y_1^2$ is $\lnot x_1$.}
    \label{fig:npcompleteness_3}
\end{figure}

Consider the instance we constructed above and suppose $\phi$ has a satisfying assignment. Using this satisfying assignment, the firefighter can construct the following strategy. The firefighter starts at $v_1$, then during rounds $2 \leq i \leq 1+n+d$ the firefighter defends whichever of $x_{i-1}$ and $\lnot x_{i-1}$ should be set to true. Note that in rounds $n+1$ through $1 + n + d$ the firefighter's choice does not actually impact the variable assignment so they can choose either of $x_{i-1}$ and $\lnot x_{i-1}$. Since every clause is satisfied by this assignment, the firefighter can then move to some vertex in $C_1$ which corresponds to a true literal, then some vertex in $C_2$ which corresponds to a true literal and so on. This strategy will save at least $1 + n + d + m$ vertices with a total weight of at least $|V(G)|^3 + (n + d)|V(G)|^2 + m|V(G)|$.

We note at this point that the only way to stop a variable vertex from burning is for the firefighter to defend the vertex, or to spend extra rounds defending the paths from $I$ to the variable vertex. It is a similar case when defending a clause vertex. Thus the firefighter will need to defend each variable vertex which they do not want to burn if they wish to save exactly half of them. Similarly, if they want to save at least $m$ clause vertices they will need to defend the $m$ clause vertices which they do not want to burn. This notion in combination with the vertex weighting will force the firefighter to defend $v_1$, exactly $n+d$ variable vertices, and at least $m$ clause vertices in order to save a total weight of at least $|V(G)|^3 + (n + d)|V(G)|^2 + m|V(G)|$. Additionally, note that the construction of $G$ is designed so that in the absence of intervention by the firefighter, if a fire breaks out at $I$, then the pair of vertices $x_i,\lnot x_i$ will be on fire at the start of round $i + 1$, and the vertices of the $j^{th}$ clause will be on fire at the start of round $1 + n + d + j$. If the firefighter defends the vertex corresponding to $x_i$ then any vertex corresponding to $x_i$ in a clause gadget $C_j$ will burn in round $2 + n + d + j$ rather than round $1 + n + d + j$.

Now suppose the firefighter can save a total weight of at least $|V(G)|^3 + (n + d)|V(G)|^2 + m|V(G)|$ when the fire breaks out at $I$. By the choice of weights, the firefighter must have saved $v_1$, and at least $n + d$ variable vertices. Thus the firefighter must start on $v_1$ since otherwise $v_1$ will burn in the first round. At the end of round $n + d + 1$ all the variable vertices will either be burned or defended, so the firefighter can defend at most $n + d$ of these. Since the firefighter cannot save more than $n+d$ variable vertices they must also save at least $m$ clause vertices. If the firefighter does not defend at least one of $x_i, \lnot x_i$ in round $i+1$ then both will burn and the firefighter will not be able to later defend any vertices in the clause $x_i \lor x_i \lor \lnot x_i$. Note that this means that in round $1 + n + d$ the firefighter must be on one of $x_{n+d}$ and $\lnot x_{n+d}$ so the only clause vertices the firefighter can defend in round $2 + n + d$ are in the first clause. Note as well that if the firefighter's choices of which variable vertices to defend does not satisfy the first clause then all vertices in the first clause are burning in round $2 + n + d$, and in this case they will not be able to defend $m$ clause vertices. Thus if the firefighter saved a weight of at least $|V(G)|^3 + (n + d)|V(G)|^2 + m|V(G)|$ they must have been able to defend a vertex in the first clause. Similarly they can only reach clause vertices in $C_2$ in round $3 + n + d$ and they can only defend one if the firefighter defended a set of variables representing a satisfying assignment for the second clause. Continuing with this notion, the firefighter must be able to defend a vertex in every clause, which implies that all clauses are satisfied and the firefighter must have defended vertices corresponding to such an assignment.
\end{proof}

We now show that the above reduction can be extended to the unweighted setting.

\begin{theorem} \label{thm:noweights}
    {\sc $d$-DPR-$1$-FF} and {\sc $d$-DR-$1$-FF} are NP-complete for $d \geq 2$.
\end{theorem}

\begin{proof}
    We present a reduction from {\sc $3$-SAT} to {\sc $d$-DPR-$1$-FF} using the image of the reduction in Theorem~\ref{thm:fixed_d_hard} as a starting point. We present a second argument to show that the same reduction works in the case where there are no path restrictions.

    Let $\phi$ be a $3$-CNF formula modified as in the proof of Theorem~\ref{thm:fixed_d_hard} and construct $G$ from $\phi$ using the same procedure as before. For each $v \in V(G)$ let $w(v)$ be the positive integer weight assigned to $v$. To produce $G'$, take each vertex $v$ and add $w(v)-1$ leaves to $v$. For a vertex $v$ which had a weight greater than $1$ in $G$, let $L_v$ be the set of leaves which were added and made adjacent to $v$.

    We first focus on the path restricted case. To see that the firefighter saves at least $|V(G)|^3 + (n + d)|V(G)|^2 + m|V(G)|$ vertices on $G'$ when a fire breaks out at $I$, we first observe that the firefighter can defend the same sequence of vertices as they did on $G$. We will now show that it is never advantageous for the firefighter to defend one of these new leaves, thus any sequence they play on $G'$ can be translated to a sequence of moves in $G$. To see this observe that in order to defend a leaf $u \in L_v$ either $v$ is not burning so the firefighter could pass through $v$ to reach $u$, or the firefighter started on $u$. If the firefighter started on $u$ then either they could have started on $v$ and defended at least one extra vertex or $v$ was an initially burning vertex. Now note that if the firefighter can only save a positive number of vertices by defending a leaf and then letting the fire spread, then it must be that every vertex is either a leaf or an initially burning vertex. This degenerate case can be recognized in polynomial time and never occurs in the reduction described in Theorem~\ref{thm:fixed_d_hard}.

    Now suppose we are in the case where there are distance restrictions but not path restrictions. Again, the firefighter can simulate their strategy from the weighted instance to save at least $|V(G)|^3 + (n+d)|V(G)|^2 + m|V(G)|$, but we must show that they gain no advantage from defending the leaves. In order to defend fewer than $n + d$ variable vertices or fewer than $m$ clause vertices the firefighter must save at least $|V(G)|^2$ or $|V(G)|$ leaves which are adjacent to a burning vertex. This requires $|V(G)|^2$ and $|V(G)|$ different rounds respectively, which is much larger than the at most $2 + n + d + m$ rounds that the game lasts. Thus the firefighter cannot save more than $2 + n + d + m$ leaves adjacent to burning vertices and so they cannot save the additional $|V(G)|$ vertices required to avoid playing the best strategy from $G$.
\end{proof}

At this point we have shown NP-completeness for $b = 1$, but we will still need to extend this to the setting with more than one firefighter to prove Theorem~\ref{thm:npcomplete}. Towards this end, let $G \boxtimes H$ be the \emph{strong product} of $G$ and $H$. This graph has vertex set $V(G) \times V(H)$ and edges between pairs if they are at distance at most $1$ in both coordinates.

\begin{proof}[\textbf{Proof of Theorem~\ref{thm:npcomplete}}]
    We present a reduction from {\sc $3$-SAT} to {\sc $d$-DR-$b$-FF} using the image of the reduction in Theorem~\ref{thm:noweights} as a starting point. This proof never requires the firefighters to move through the fire so this is also a valid reduction for the path restricted case.
    
    Let $\phi$ be a $3$-CNF formula modified as in the proof of Theorem~\ref{thm:fixed_d_hard}, $G$ be constructed using the same procedure as in the proof of Theorem~\ref{thm:fixed_d_hard}, $G'$ be as in Theorem~\ref{thm:noweights} and $G''$ be $G' \boxtimes K_b$. We claim that the $b$ firefighters can save $b(|V(G)|^3 + (n + d)|V(G)|^2 + m|V(G)|)$ vertices on $G''$ if and only if the firefighter on $G'$ can save $|V(G)|^3 + (n + d)|V(G)|^2 + m|V(G)|$ vertices. First we note that if the firefighter defends $v$ in $G'$ then the $b$ firefighters can defend the copy of $K_b$ in $G''$ to have the same effect in terms of how the fire will spread. Similarly, if the firefighter moves from $u$ to $v$ in $G'$ then the $b$ firefighters can move from the copy of $K_b$ corresponding to $u$ to the copy of $K_b$ corresponding to $v$, as long as they have been following this mirroring strategy the whole game. Thus if the firefighter can save $|V(G)|^3 + (n + d)|V(G)|^2 + m|V(G)|$ vertices on $G'$ then the $b$ firefighters can save $b(|V(G)|^3 + (n + d)|V(G)|^2 + m|V(G)|)$ vertices on $G''$.

    Now suppose the $b$ firefighters can save $b(|V(G)|^3 + (n + d)|V(G)|^2 + m|V(G)|)$ vertices on $G''$. If this is the case then all the firefighters must have started on the copy of $K_b$ corresponding to $v_1$ since otherwise the fire will spread to at least one vertex in the clique and then to the $b(|V(G)|^3 - 1)$ leaves adjacent to the clique. In round $2 \leq i \leq n + 1$ if the firefighters do not completely defend a clique corresponding to $x_{i-1}$ or $\lnot x_{i-1}$ then once again the fire will spread to at least one vertex in the clique corresponding to $x_{i-1}$ and at least one vertex in the clique corresponding to $\lnot x_{i-1}$. As a result, the $b(|V(G)|^2 - 1)$ leaves adjacent to each of those cliques will burn. Additionally, if $z \in \{x_{i-1},\lnot x_{i-1}\}$ appears in the $j^{th}$ clause of $\phi$, then the corresponding clique will burn in round $1 + n + d + j$ instead of round $2 + n + d + j$. Thus the firefighters must always defend a clique corresponding to a literal in these rounds and they must defend exactly one of the cliques corresponding to $x_{i-1}$ and $\lnot x_{i-1}$ in round $2 \leq i \leq n + 1$. Note that this implies that the firefighters have created a truth assignment of the variables which we will use later to show that $\phi$ has a satisfying assignment. Since the firefighters must defend $n + d$ of the cliques corresponding to the variable vertices, they must also defend cliques corresponding to the dummy variables in the next $d$ rounds. Thus at the end of round $1 + n + d$ the firefighters must be on a clique corresponding to either $x_{n+d}$ or $\lnot x_{n+d}$. Now the only clause vertices the firefighters can reach in round $2 + n + d$ are in the first clause. The firefighters must still save a weight of $bm|V(G)|$ and they must do so over the course of the next $m$ rounds since at that point the fire will no longer be able to spread and the firefighters will no longer be able to defend any new vertices. The firefighters cannot save more than $b|V(G)|$ vertices in each of these rounds so they must save $b|V(G)|$ vertices in each round. The only way to do this is in round $1 + n + d + j$ is to have all the firefighters defend one of the cliques corresponding to a literal in the $j^{th}$ clause, and this is only possible if the clause is satisfied by the previously mentioned assignment of the variables. Thus each clause is satisfied and so there must be a satisfying assignment.
\end{proof}

A natural question which arises here is which instances of {\sc $d$-DPR-$b$-FF} and {\sc $d$-DR-$b$-FF} are tractable. It is known that {\sc $b$-Firefighter} is NP-complete even when restricted to trees of maximum degree $b + 2$; however for {\sc $d$-DPR-$b$-FF} and {\sc $d$-DR-$b$-FF}, the problem is much easier to solve on trees.

\begin{lemma} \label{lemma:polytrees}
    {\sc $d$-DPR-$b$-FF} is solvable in logarithmic space on trees whenever $|F| = 1$ regardless of the value of $d$. Similarly, {\sc $d$-DR-$b$-FF} is solvable in logarithmic space on trees whenever $|F|=1$ and $d \leq 2$.
\end{lemma}

\begin{proof}
    Let $T$ be a tree and root the tree at the initial burning vertex $u$. Label the neighbours of $u$ as $v_1,\ldots, v_{d(u)}$. Define the quantities $S_{v_i}$ to be the number of vertices in the component of $T \setminus u$ containing $v_i$. We claim that the best possible strategy is to defend the $b$ vertices in $\{v_1,\ldots,v_{d(u)}\}$ which have the largest $S_{v_i}$ values. To see this, we first note that any firefighter can only defend vertices in one component of $T \setminus u$ since moving to another component would require moving through $u$ and a burning neighbour of $u$. Now, we see that if the firefighter defends one of these $v_i$ then they save all vertices in that component of $T \setminus u$ and since they certainly cannot do better than saving all of the vertices in the chosen component this will always be optimal. 

    We now show that the $b$ largest $S_{v_i}$ can be calculated in logarithmic space. First, note that for a fixed $i$, $S_{v_i}$ can be computed by iterating over all vertices $w \in V(T) \setminus \{u\}$ and using Reingold's algorithm~\cite{reingold2008logspace} to check if $v_i$ and $w$ are in the same component of $T \setminus u$. Since the forest $T \setminus u$ can be computed by a logspace transducer, $S_{v_i}$ can be computed in logarithmic space. The remainder of the algorithm is as follows. For $1 \leq i \leq b$, compute $S_{v_i}$ and store it in a sorted list. For $b + 1 \leq i \leq d(u)$, compute $S_{v_i}$, insert it into the list of length $b$ so that the list remains sorted, then delete the smallest element in the list. Since the list is of constant length and integer comparison is trivially solvable in logarithmic space, this does not violate our logarithmic space bound. The algorithm uses at most $(b + 1) \log |V(T)|$ space to store the computed $S_{v_i}$'s and uses $\mathcal{O}(\log |V(T)|)$ space to compute each $S_{v_i}$ and reuses this space for each $v_i$. Thus the algorithm runs in logarithmic space.
\end{proof}

Note that it is widely believed that L (the set of problems decidable in logarithmic space), is a strict subset of P\footnote{Much stronger separations are believed to be true, see~\cite{greenlaw1995parallel} for example.} and that P is a strict subset of NP, so {\sc $d$-DPR-$b$-FF} being in L on trees while {\sc $b$-Firefighter} is NP-complete represents a significant difference in their complexities.

\section{Integer Programming Solutions}\label{sec:integerprogram}

A natural question to ask about any NP-complete problem is whether there are ways to solve it using `relatively efficient' methods. Solving these problems by transforming them into something that can be passed as input to an integer program solver or a SAT solver allows us to take advantage of very well optimized pieces of software like HiGHS~\cite{huangfu2018parallelizing}. Of course, since our problem is in NP such a method for solving our problem certainly exists, but using something like a more general reduction via {\sc $3$-SAT} could make it more difficult to relax specific conditions of the problem. For example, since the role of each constraint will be made clear here, we can relax the constraints involving vertices being defended as in~\cite{hartke2006attempting} and keep the constraints involving the distance restrictions as integer constraints in order to obtain a mixed integer program for a fractional version of the distance-restricted firefighter problem. Having access to a method for finding solutions where we can easily recover the firefighter strategy also allows us to formulate conjectures such as Conjecture~\ref{con:18grid} which are not necessarily obvious without solving many instances of {\sc $d$-DR-$b$-FF}. 

The original {\sc MVS} function was shown in~\cite{develin2007fire} to be computable by a $0,1$-integer program which is depicted in Figure~\ref{fig:develinhartkeprogram}. The program encodes whether or not a vertex is burned at a particular point in time using the $b_{x,t}$ variables and similarly encodes whether or not a vertex is defended using the $d_{x,t}$ values. Note that we use a capital $B$ here to denote the number of firefighters so it is disambiguated from the $b_{x,t}$ values and we use $F$ to denote the set of initially burning vertices. 

    \begin{figure}[H]
        \centering 
    \begin{align*}
        &\text{Minimize $\sum_{x \in V(G)} b_{x,T}$, subject to:}\\
        &b_{x,t} \geq b_{x,t-1} && x \in V(G), 1\leq t \leq T\\
        &d_{x,t} \geq d_{x,t-1} && x \in V(G), 1\leq t \leq T\\
        &b_{x,t} + d_{x,t} \geq b_{y,t-1} && x \in V(G), y \in N(x), 1 \leq t \leq T \\
        &b_{x,t} + d_{x,t} \leq 1 && x \in V(G), 1 \leq t \leq T \\
        &\sum_{y \in N(x)} b_{y,t-1} \geq b_{x,t} && x \in V(G), 1 \leq t \leq T \\
        &\sum_{x \in V(G)} (d_{x,t} - d_{x,t-1}) \leq B && 1 \leq t \leq T \\
        &d_{x,0} = 0 && x \in V(G) \\
        &b_{x,0} =   \begin{cases}
                        1 & \text{ if $x \in F$}\\
                        0 & \text{ otherwise}
                    \end{cases} && x \in V(G) \\
        &b_{x,t},d_{x,t} \in \{0,1\} && x \in V(G), 0 \leq t \leq T
    \end{align*}
    \caption{A $0,1$-integer program for solving firefighter {\sc MVS}, due to Develin and Hartke~\cite{develin2007fire}.}
    \label{fig:develinhartkeprogram}
\end{figure}

    While this program can compute the original {\sc MVS} function, its encoding of the information loses certain features that are needed to impose distance and path restrictions on the problem. In particular, if a firefighter moves onto a previously defended vertex, then the set of defended vertices does not change, and that firefighter's current position is now unknown. Thus on the next turn it is impossible to validate that firefighter's move since it is impossible to know exactly which vertex the firefighter was on.

    In order to construct an extension to the Develin and Hartke program that computes {\sc DRMVS}, we first have to re-encode the firefighters' positions in a way that solves the aforementioned information loss problem, but also in a way which lends itself well to linear transformations. Our solution is to encode the firefighters' positions on any given turn as a one-hot\footnote{A \textit{one-hot} encoding refers to a vector where all entries but one are zero and the remaining entry is a $1$.} encoded vector of length $\binom{|V(G)| + B - 1}{B}$, as that is the number of weak compositions\footnote{A \textit{weak composition} of $n$ into $k$ parts is an ordered sum of $k$ non-negative integers which sum to $n$.} of $B$ into $|V(G)|$ parts and we can order the compositions lexicographically. Note that we consider $B$ to be a constant and so the quantity $\binom{|V(G)| + B - 1}{B}$ is only polynomially larger than $|V(G)|$. Specifically, in the case where $B = 1$, $\binom{|V(G)| + B - 1}{B} = |V(G)|$ and so one position in the vector has value $1$ while the rest have value $0$. We use this approach because weak compositions are in correspondence with positions the firefighters can occupy since each element of the composition corresponds to a vertex of the graph and the quantity corresponds to the number of firefighters on that vertex. Thus if the firefighters' positions at time $t$ correspond to the $i^{th}$ weak composition we represent the position with the one-hot encoded vector with a $1$ in the $i^{th}$ coordinate and we label this vector $p_t$. 

    As an example, consider a path on three vertices and all six possible placements of two firefighters on this graph sorted in lexicographical order:

    \[
        \begin{bmatrix} 2\\ 0\\ 0\end{bmatrix}, \begin{bmatrix} 1\\ 1\\ 0\end{bmatrix}, \begin{bmatrix} 1\\ 0\\ 1\end{bmatrix}, \begin{bmatrix} 0\\ 2\\ 0\end{bmatrix}, \begin{bmatrix} 0\\ 1\\ 1\end{bmatrix}, \begin{bmatrix} 0\\ 0\\ 2\end{bmatrix}.
    \]

    Thus if there are two firefighters on the second vertex and no firefighters on the other two vertices we would represent the firefighters' position as the vector with a $1$ in the fourth position and $0$'s in all other positions.

    We now need to make these position vectors force the $d_{x,t}$ to take on their corresponding values. The first step in doing this is to extract which vertices are occupied from the one-hot encoded position vector, which we do by constructing a matrix $A$ with $|V(G)|$ rows and $\binom{|V(G)| + B - 1}{B}$ columns where the $i^{th}$ entry in the $j^{th}$ row is $1$ if the $j^{th}$ entry of the weak composition corresponding to the $i^{th}$ position is positive and the entry is $0$ otherwise. Using our previous example this would give the matrix:\[
        A = \begin{bmatrix} 1 & 1 & 1 & 0 & 0 & 0\\ 0 & 1 & 0 & 1 & 1 & 0\\ 0 & 0 & 1 & 0 & 1 & 1\end{bmatrix}.
    \]

    Multiplication by $A$ takes a position vector and yields a binary vector where the $i^{th}$ entry being $1$ implies that the $i^{th}$ vertex is occupied, and thus the corresponding $d_{x,t}$ needs to be $1$ as well. In order to set the $d_{x,t}$ value we want to take the maximum of all the entries corresponding to vertex $x$ in the vectors $A p_\tau$ where $\tau \leq t$. This is achieved in two parts. First set the sum of all these values as an upper bound on $d_{x,t}$ since if all of these values are zero then $d_{x,t}$ will also be zero. Secondly set each of the values individually as a lower bound on $d_{x,t}$, then if any of the values are $1$, $d_{x,t}$ will be $1$ as well.

    So we can now associate a one-hot encoded position vector to each time step, and this can be used to construct a bilinear form which determines if the move from the positions at time $t$ to the positions at time $t + 1$ is valid. Define $M_d$ to be a $\binom{|V(G)| + B - 1}{B}$ square matrix where the $i^{th}$ entry in the $j^{th}$ row is $0$ if there is a valid move from position $i$ to position $j$ and $1$ if that is not the case. Using our previous example with $d = 1$ we get the following matrix:\[
        M_1 = \begin{bmatrix} 1 & 1 & 0 & 1 & 0 & 0\\ 1 & 1 & 1 & 1 & 1 & 0\\ 0 & 1 & 1 & 1 & 1 & 0\\ 1 & 1 & 1 & 1 & 1 & 1\\ 0 & 1 & 1 & 1 & 1 & 1\\ 0 & 0 & 0 & 1 & 1 & 1\end{bmatrix}.
    \]

    This allows us to use $p_{t}^{T}M_{d}p_{t+1}$ to determine if the given move is valid, and we can in fact replace the bilinear form $M_d$ with a dot product by instead encoding the vector $a_t = p_t \otimes p_{t+1}$ for each time step and noting that we can recover $p_t$ from $a_t$ using a projection map $\mu$. This does also now require us to ensure that the projection to the second factor of the tensor product in $a_t$ and the projection to the first factor of the tensor product in $a_{t+1}$ are in fact equal. Figure~\ref{fig:programextension} depicts the extension to the Develin and Hartke program which computes {\sc DRMVS}.

    \begin{figure}[H]
    \begin{minipage}{0.45\textwidth}
    \begin{align*}
        N &= \binom{n + B - 1}{B}\\
        p_t &\in \mathbb{F}_2^N\\
        p_t^i &= \begin{cases}
            1 & \text{If $S_i$ is occupied at time $t$}\\
            0 & \text{Otherwise}\\
            \end{cases}\\
        a_t &= p_t \otimes p_{t+1} \in \mathbb{F}_2^N \otimes \mathbb{F}_2^N\\
        \mu_1 &: \mathbb{F}_2^N \otimes \mathbb{F}_2^N \to \mathbb{F}_2^N\\
        \mu_1(a_t) &= p_t\\
        \mu_2 &: \mathbb{F}_2^N \otimes \mathbb{F}_2^N \to \mathbb{F}_2^N\\
        \mu_2(a_t) &= p_{t+1}\\
        M_d &: \mathbb{F}_2^N \otimes \mathbb{F}_2^N \to \{0,1\}\\
        M_d(p_t \otimes p_{t+1}) &= \begin{cases}
          0 & \text{If there is a valid move}\\
          1 & \text{Otherwise}\\
        \end{cases}\\
    \end{align*}
\end{minipage}
    \begin{minipage}{0.45\textwidth}
    \begin{align*}
        \mathbb{1} \cdot a_{t'}^i &= 1\text{, }t' > 1\\
        \mathbb{1} \cdot a_{0}^i &= 0\\
        \mathbb{1} \cdot p_t^i &= 1\\
        \mu_2(a_t) &= \mu_1(a_{t+1})\\
        M_d(a_t) &\leq 0\\
        d_{x,t} &\geq \mu_1(a_{t'})[x]\text{, } t' \leq t\\
        d_{x,t} &\leq \sum_{t' \leq t} \mu_1(a_{t'})[x]\\    
    \end{align*}
\end{minipage}
\caption{An extension to the $0,1$-integer program which computes {\sc DRMVS}.}
\label{fig:programextension}
\end{figure}

    We can see that this implements all the pieces discussed in the previous paragraphs. The $a_t$ values are our tensor products of position vectors, the $\mu_i$ projection maps ensure that $a_t$ and $a_{t+1}$ are compatible, the transformation $M_d$ validates moves, and the last two lines force $d_{x,t}$ to take on the values corresponding to $p_t$.

    It is relatively straightforward to construct a program which takes a {\sc DRMVS} problem instance as input and outputs the corresponding linear program. We used such a program as well as the HiGHS solver~\cite{huangfu2018parallelizing} to solve a $13 \times 13$ square grid (i.e. $P_{13}\:\square\:P_{13}$) with $d=3$ and a single firefighter with a runtime of approximately $28$ days on a $6$-core Zen+ desktop platform. We also used this to solve a large number of instances on square grids with $d=1$ which led to Lemma~\ref{lem:18grid} and the related Conjecture~\ref{con:18grid}. Note that we label the vertices using the set $[n] \times [n]$ and $(i,j)$ is adjacent to $(x,y)$ if $i = x$ and $|j - y| = 1$ or vice versa.

    \begin{lemma} \label{lem:18grid}
        \[\lim_{n \to \infty} \frac{\textsc{DRMVS}\left(P_n \:\square\: P_n,\left(\left\lceil\frac{n}{2}\right\rceil,\left\lceil\frac{n}{2}\right\rceil\right);1,1\right)}{n^2} \geq \frac18\]
    \end{lemma}

    \begin{proof}
        To show this, we show $\textsc{DRMVS}\left(P_n \:\square\: P_n,\left(\left\lceil\frac{n}{2}\right\rceil,\left\lceil\frac{n}{2}\right\rceil\right);1,1\right) \geq \left\lceil \frac{\left\lceil \frac{n}{2}\right\rceil - 1}{2}\right\rceil \left\lceil\frac{n}{2}\right\rceil$. Consider the strategy where the firefighter starts at $\left(\left\lceil\frac{n}{2}\right\rceil,1\right)$ and defends the vertices: 
        \[\left(\left\lceil\frac{n}{2}\right\rceil,2\right), \left(\left\lceil\frac{n}{2}\right\rceil,3\right), \ldots, \left(\left\lceil\frac{n}{2}\right\rceil,\left\lceil \frac{\left\lceil \frac{n}{2}\right\rceil - 1}{2}\right\rceil\right).\] 
        At this point the vertex $\left(\left\lceil\frac{n}{2}\right\rceil,\left\lceil \frac{\left\lceil \frac{n}{2}\right\rceil - 1}{2}\right\rceil + 1\right)$ is burning and the firefighter now defends the vertices: 
        \[\left(\left\lceil\frac{n}{2}\right\rceil + 1,\left\lceil \frac{\left\lceil \frac{n}{2}\right\rceil - 1}{2}\right\rceil\right), \left(\left\lceil\frac{n}{2}\right\rceil + 2,\left\lceil \frac{\left\lceil \frac{n}{2}\right\rceil - 1}{2}\right\rceil\right), \ldots, \left(n,\left\lceil \frac{\left\lceil \frac{n}{2}\right\rceil - 1}{2}\right\rceil\right),\] 
        which saves the desired number of vertices and the result follows.
    \end{proof}

    We conjecture that the lower bound in Lemma~\ref{lem:18grid} is the exact value of the limit.

    \begin{conjecture} \label{con:18grid}
        \[\lim_{n \to \infty} \frac{\textsc{DRMVS}\left(P_n \:\square\: P_n,\left(\left\lceil\frac{n}{2}\right\rceil,\left\lceil\frac{n}{2}\right\rceil\right);1,1\right)}{n^2} = \frac18\]
    \end{conjecture}

    It would also be possible to extend this program to compute {\sc DPRMVS}; however the size of the program would grow much more quickly and as such we only give a high level overview. The idea is that for each pair of positions and each value of $d$ there is some finite set of `connectors', which are sets of vertices the firefighters can move through to get from one position to the other. By encoding all these connector sets for a pair of positions in a binary vector, and also encoding which of them contain a burning vertex in another vector, their dot product can encode whether or not any of the sets are unburned. This can all be done using similar mapping techniques as with the linear program for computing {\sc DRMVS} and the connectors can be computed in polynomial time using a variation of Djikstra's algorithm~\cite{djikstra1959algorithm}.

\section{Expected Damage}\label{sec:expectedburn}

Another area of interest for firefighting is determining the \textit{expected damage} sustained by a graph given a fixed number $q$ of initial burning vertices and a fixed number $b$ of firefighters. This is equivalent to computing the \textit{surviving rate} defined in~\cite{cai2009surviving}. The expected damage is denoted as $E(G;q,b)$ and is computed as:
    
    \begin{equation*}
        E(G;q,b) = \sum_{F \subset V(G)\text{, }|F|=q} \frac{|V(G)| - \text{{\sc MVS}}(G,F;b)}{\binom{|V(G)|}{q}}
    \end{equation*}

    The expected damage can be extended to the distance-restricted variant by replacing {\sc MVS} by {\sc DRMVS}, and to the distance path restricted variant by replacing {\sc MVS} by {\sc DPRMVS}. We denote these expected damage functions by $E(G;q,b,d)$ and $E_{pr}(G;q,b,d)$, respectively.

    A classic problem for expected damage is to determine which connected graphs of order $n$ minimize the expected damage~\cite{crosby2005designing,finbow2000minimizing}. These are typically referred to as \textit{optimal graphs}. For the original problem, the optimal graphs with a single fire and firefighter were completely characterized in~\cite{finbow2000minimizing} and the optimal graphs with two fires and two firefighters were completely characterized in~\cite{crosby2005designing}. For a single fire and a single firefighter the optimal graphs are stars and caterpillars where any vertices of degree at least three are at least a distance of three apart. Since the optimal strategy on $K_{1,n-1}$ only involves a single placement of the firefighter we can achieve the same minimum in the distance-restricted and distance-path-restricted games. Thus with a single fire and a single firefighter we can narrow our search for optimal graphs from all graphs to graphs which are optimal in the original game. We will soon see that this is not possible with two fires and two firefighters, and in fact determining the optimal graphs when there are distance and/or path restrictions can be much more difficult than in the original version of firefighter.

    \begin{theorem}
        When $d \leq 2$ or there are path restrictions the optimal graphs with one fire and one firefighter are stars and graphs formed by taking a star and adding leaves to one leaf of the star.
    \end{theorem}

    \begin{proof}
        First we show that the optimal expected value is $2-\frac2n$. Consider the star on $n$ vertices. When the fire starts on one of the $n-1$ leaves only one vertex burns and when the fire starts in the center $n-1$ vertices burn. Thus in total $\left(n-1\right) + \left(n-1\right)$ vertices burn, and we then divide by $\binom{n}{1}$ to get that the expected damage is $\frac{\left(n-1\right) + \left(n-1\right)}{n} = 2-\frac2n$. This also works for $d \leq 2$ or if there are path restrictions. This shows that the optimal graphs with path restrictions or $d \leq 2$ must be contained in the set of optimal graphs for the original game. This is true since we certainly cannot have a smaller expected damage in our variants of the game, but as we have just shown we can achieve the same value.

        We now show that the remaining graphs from the theorem statement are optimal. Suppose we have an edge with $\ell$ leaves on one end of the edge and $t$ leaves on the other end. If the fire starts on one of the $\ell + t$ leaves only one vertex burns, if it starts on one of the higher degree vertices either $\ell + 1$ or $t + 1$ vertices burn. Thus we get $\ell + t + \ell + 1 + t + 1$ vertices burning in total, and $\ell + t = n - 2$ so we get $2\left(n-2\right) + 2 = 2n - 2$ burning in total. We take that and divide by $n$ to get $2 - \frac2n$.

        Now we show that any other caterpillar graph is non-optimal. Here we take advantage of the fact that since the firefighters are either path-restricted or can only move a distance less than or equal to $2$ they can only defend once and then let the fire spread. This is similar to the technique used in the proof of Lemma~\ref{lemma:polytrees}. First it is necessary to define some quantities relating to our caterpillars. First we require that the spine of our caterpillar only contains vertices of degree greater than $1$ and we define the number of vertices on the spine to be $p$. Now define the number of vertices on the spine with a neighbour not on the spine to be $R$ and for each such vertex let the number of non-spine neighbours be $r_i$. Note that based on our definition of the spine the case of $p=1$ corresponds to stars and the case of $p=2$ corresponds to the modified stars mentioned in the theorem statement so we need not concern ourselves with either of these cases.

        \textbf{Case 1: $R=2$, $p>2$:} In this case, our caterpillar is such that the two end vertices of the spine are the only spine vertices that can have non-spine neighbours. 

\begin{figure}[H]
    \centering
\begin{tikzpicture}[scale=0.5]
    \SetUpVertex[FillColor=white]

    \tikzset{VertexStyle/.append style={minimum size=8pt, inner sep=1pt}}

    \Vertex[x=0,y=0,NoLabel=true,]{V0}
    \Vertex[x=2,y=0,NoLabel=true,]{V1}
    \Vertex[x=4,y=0,NoLabel=true,]{V2}
    \Vertex[x=6,y=0,NoLabel=true,]{V3}
    \Vertex[x=8,y=0,NoLabel=true,]{V4}

    \Vertex[x=-2,y=1,NoLabel=true,]{V0A}
    \Vertex[x=-2,y=-1,NoLabel=true,]{V0B}
    
    \Vertex[x=10,y=1,NoLabel=true,]{V4A}
    \Vertex[x=10,y=-1,NoLabel=true,]{V4B}
    
    \Edge(V0)(V1)
    \Edge(V1)(V2)
    \Edge(V2)(V3)
    \Edge(V3)(V4)

    \Edge(V0)(V0A)
    \Edge(V0)(V0B)
    \Edge(V4)(V4A)
    \Edge(V4)(V4B)

\end{tikzpicture}
    \caption{An example of a caterpillar with $R = 2$, $p = 5$}
\end{figure}

        If $p = 3$, $r_1 = \ell$ and $r_2 = t$ with $\ell \geq t$ we can see that at least $\ell + t + \ell + 1 + t + 1 + t + 2$ vertices will burn. Using the fact that $\ell + t = n-3$ we get $2\left(n-3\right) + 4 + t$, which, when divided by $n$, yields $2 - \frac2n + \frac{t}{n} > 2 - \frac2n$.

        If $p = 4$, then let $\ell = r_1 \geq r_2 = k$ again. We again see that at least $\ell + t + \ell + 1 + t + 1 + t + 2 + t + 2$ vertices will burn, we divide by $n$ and use the fact that $\ell + t = n-4$ to get $2 - \frac2n + \frac{2k}{n} > 2 - \frac2n$.

        If $p > 4$ then note that there are $n-p$ non-spine vertices. From the leaves we have $n-p$ vertices burning, from the ends of the spine we get another $n-p+2$ burning, from the interior vertices of the spine we get $2\left(p-2\right)$ vertices burning since that vertex as well as one of its spine neighbours will burn, and from the vertices on the spine that are at distance $2$ or more from the both ends of the spine (which are guaranteed to exist since $p > 4$) we get an additional spine vertex burning thus giving us an additional $p - 4$ vertices burned. Adding these all up gives $2n + p - 6$, dividing by $n$ gives $2 + \frac{p-6}{n}$ which is greater than $2 - \frac2n$ since $p > 4$.

\begin{figure}[H]
    \centering
\begin{tikzpicture}[scale=0.5]
    \SetUpVertex[FillColor=white]

    \tikzset{VertexStyle/.append style={minimum size=8pt, inner sep=1pt}}

    \Vertex[x=0,y=0,NoLabel=true,]{V0}
    \Vertex[x=2,y=0,NoLabel=true,]{V1}
    \Vertex[x=4,y=0,NoLabel=true,]{V2}
    \Vertex[x=6,y=0,NoLabel=true,]{V3}
    \Vertex[x=8,y=0,NoLabel=true,]{V4}

    \Vertex[x=-2,y=1,NoLabel=true,]{V0A}
    \Vertex[x=-2,y=-1,NoLabel=true,]{V0B}
    
    \Vertex[x=10,y=1,NoLabel=true,]{V4A}
    \Vertex[x=10,y=-1,NoLabel=true,]{V4B}

    \Vertex[x=4.5,y=-1,NoLabel=true,]{V2A}
    \Vertex[x=3.5,y=-1,NoLabel=true,]{V2B}

    \Edge(V0)(V1)
    \Edge(V1)(V2)
    \Edge(V2)(V3)
    \Edge(V3)(V4)

    \Edge(V2)(V2A)
    \Edge(V2)(V2B)
    
    \Edge(V0)(V0A)
    \Edge(V0)(V0B)
    \Edge(V4)(V4A)
    \Edge(V4)(V4B)

\end{tikzpicture}
    \caption{An example of a caterpillar with $R = 3$, $p = 5$}
\end{figure}

    \textbf{Case 2: $R > 2$, $p > 2$:} First, each of the non-spine vertices gets burned once when they are a source and once when their nearest spine vertex is a source (which also burns that spine vertex), which gives $2\left(n-p\right) + R$ burned vertices. Second, every spine vertex with non-spine neighbours except for the outermost such vertices will cause another spine vertex with non-spine neighbours to burn, thus burning at least an additional $\left(R-2\right)\min\{r_i\}$ vertices. Every spine vertex with all their neighbours on the spine will always burn a spine vertex with neighbours off the spine since we forced the endpoints of the spine to have neighbours off the spine, thus burning at least an additional $\left(p-R\right) \min\{r_i\}$ vertices. Finally, every spine vertex that is not an endpoint of the spine (which is guaranteed to exist since $p > 2$) will burn an additional vertex on the spine, thus burning an additional $p-2$ vertices. This yields $2n - 2p + R + \left(R-2\right) \min\{r_i\} + \left(p-R\right) \min\{r_i\} + p-2$ which is greater than or equal to $2n - 2p + R + R-2 + p-R + p-2 = 2n + R - 4$. Now dividing by $n$ gives $2 + \frac{R-4}{n}$ which is greater than $2 - \frac2n$ since $R > 2$.
    \end{proof}

    As previously mentioned, we cannot use the same strategy to find the optimal graphs with two fires and two firefighters with distance-restrictions and path-restrictions. This is because the optimal graphs of order at least $14$ are stars with four of their original edges subdivided, and the expected damage is strictly larger with $d \leq 2$ or with path restrictions compared to the original game on these graphs. Simply consider the case where one fire starts on the central vertex and the other starts on an adjacent leaf vertex. The optimal strategy for the original game is to defend two of the non-leaf vertices adjacent to the center and then defend the two leaf vertices which have not yet burned but are about to burn. This strategy is depicted in Figure~\ref{fig:22optimal}. The orange diamond vertices represent the initial fires, the black star shaped vertices represent the vertices defended on the first turn.
 
\begin{figure}[H]
    \centering
\begin{tikzpicture}[scale=0.5]
    \SetUpVertex[FillColor=white]

    \tikzset{VertexStyle/.append style={minimum size=8pt, inner sep=1pt}}

    \Vertex[x=0,y=0,NoLabel=true,]{V00}
    
    \Vertex[a=0,d=2cm,NoLabel=true,]{V10}
    \Vertex[a=60,d=2cm,NoLabel=true,]{V11}
    \Vertex[a=120,d=2cm,NoLabel=true,]{V12}
    \Vertex[a=180,d=2cm,NoLabel=true,]{V13}
    \Vertex[a=240,d=2cm,NoLabel=true,]{V14}
    \Vertex[a=300,d=2cm,NoLabel=true,]{V15}
    
    \Vertex[a=120,d=4cm,NoLabel=true,]{V22}
    \Vertex[a=180,d=4cm,NoLabel=true,]{V23}
    \Vertex[a=240,d=4cm,NoLabel=true,]{V24}
    \Vertex[a=300,d=4cm,NoLabel=true,]{V25}

    \Edge(V00)(V10)
    \Edge(V00)(V11)
    \Edge(V00)(V12)
    \Edge(V00)(V13)
    \Edge(V00)(V14)
    \Edge(V00)(V15)
    
    \Edge(V12)(V22)
    \Edge(V13)(V23)
    \Edge(V14)(V24)
    \Edge(V15)(V25)
    
    \tikzset{VertexStyle/.append style={minimum size=12pt, inner sep=0.5pt, diamond, orange, text=black}}
    
    \Vertex[x=0,y=0,L={\scriptsize 0},]{V00}
    \Vertex[a=0,d=2cm,L={\scriptsize 0},]{V10}
    
    \tikzset{VertexStyle/.append style={minimum size=8pt, inner sep=1pt, circle, orange, text=black}}
   
    \Vertex[a=60,d=2cm,L={\scriptsize 1},]{V11}
    \Vertex[a=120,d=2cm,L={\scriptsize 1},]{V12}
    \Vertex[a=180,d=2cm,L={\scriptsize 1},]{V13}

    \tikzset{VertexStyle/.append style={minimum size=12pt, inner sep=0.5pt, star, black, text=black!20}}

    \Vertex[a=240,d=2cm,L={\scriptsize 0},]{V14}
    \Vertex[a=300,d=2cm,L={\scriptsize 0},]{V15}
    
    \tikzset{VertexStyle/.append style={minimum size=8pt, inner sep=1pt, rectangle, black, text=black!20}}

    \Vertex[a=120,d=4cm,L={\scriptsize 1},]{V22}
    \Vertex[a=180,d=4cm,L={\scriptsize 1},]{V23}
    
\end{tikzpicture}
    \caption{An optimal strategy in the original game which is not feasible with path restrictions or small $d$}
    \label{fig:22optimal}
\end{figure}

The above strategy is not possible with path restrictions. As a result we cannot match the minimum from the original game and so there could be graphs whose expected damage with path restrictions sits between the expected damage without path restrictions and the expected damage with path restrictions of the graph in Figure~\ref{fig:22optimal}. Therefore the optimal graph with path restrictions need not be the same as the optimal graph without path restrictions.

A useful tool for finding optimal graphs in the original game is the fact that the {\sc MVS} function is monotonic on spanning subgraphs so any search can be restricted to trees. However, if there are restrictions on the movement of the firefighters then the deletion of edges can harm the firefighters as we saw in Lemma~\ref{lemma:polytrees} where the firefighters only get to defend for one turn. Even if the firefighters can move any distance they wish but are still not allowed to pass through burning vertices they can lose access to parts of the graph when the subgraph induced by the unburned vertices is disconnected. The remainder of the section will focus on this monotonicity question and whether or not the increase in expected damage can be bounded.

First we consider the path-restricted $d=\infty$ case with one fire and one firefighter on a cycle on $7$ vertices as well as on a path which spans that cycle. On the cycle it is easy to see that the expected damage is $2$ since the firefighter will defend a vertex adjacent to the fire, the fire will spread to its other neighbour, and the firefighter will then move around the cycle and stop the fire from spreading further. Contrast this with the path, where the optimal strategy is to defend whichever side of the fire contains a longer subpath and then let the fire spread until it cannot spread anymore. The only cases on the path where less than two vertices burn is when the fire is on a leaf, but if we pair those up with the cases where the fire starts distance $2$ from a leaf they average out to $2$ vertices burning. Thus since in every other case at least two vertices burn, and there are guaranteed to be some remaining cases where at least four burn, the path will have higher expected damage with path restrictions than without.

    So what about the case of interest with two fires and two firefighters? It turns out that for $b$ fires and $b$ firefighters, there exist graphs where the expected damage of a graph is strictly less than the expected damage of one of its connected spanning subgraphs.

    \begin{theorem} \label{thm:nonmono}
        Let $b$ be a positive integer, $G_A$ a connected graph on at least $b + 1$ vertices and $G_B$ be a connected graph on exactly $b + 1$ vertices. For any integer $\ell > 0$, create $G_{\ell}$ as follows. Label the vertices of $G_B$ $v_1, v_2, \dots, v_{b+1}$ and arbitrarily select a $\left(b + 1\right)$-subset of the vertices of $G_A$ to label using the same set of labels. For each vertex in $G_B$ connect it via a path with $\ell$ internal vertices to the vertex in $G_A$ with the same label. If $\ell$ is sufficiently large then $E_{pr}\left(G_{\ell};b,b,\infty\right) < E_{pr}\left(G_{\ell} - E\left(G_B\right);b,b,\infty\right)$. Furthermore, the ratio $\frac{E_{pr}\left(G_{\ell};b,b,\infty\right)}{E_{pr}\left(G_{\ell} - E\left(G_B\right);b,b,\infty\right)}$ can be arbitrarily large.
    \end{theorem}

    For an example of the construction in the theorem statement see Figure~\ref{fig:G3}.

    \begin{proof}
        First we show that $\binom{|V\left(G_{\ell}\right)|}{b} E_{pr}\left(G_{\ell};b,b,\infty\right) \in \mathcal{O}\left(\ell^b\right)$. Notice that $|V\left(G_{\ell}\right)| \in \mathcal{O}\left(\ell\right)$ and that there are $\mathcal{O}\left(\ell^{b-1}\right)$ initial burning configurations where at least one vertex is in $V\left(G_A\right) \cup V\left(G_B\right) \cup N\left(V\left(G_A\right)\right) \cup N\left(V\left(G_B\right)\right)$ where $N\left(S\right)$ is the set of vertices which are not in $S$ but have a neighbour in $S$. Thus since at most $\mathcal{O}\left(\ell\right)$ vertices can burn in these $\mathcal{O}\left(\ell^{b-1}\right)$ scenarios, we have that this set of initial configurations contributes at most $\mathcal{O}\left(\ell^b\right)$ to $\binom{|V\left(G_{\ell}\right)|}{b} E_{pr}\left(G_{\ell};b,b,\infty\right)$. The remaining initial configurations are configurations where no initial burning vertices are in $V\left(G_A\right) \cup V\left(G_B\right) \cup N\left(V\left(G_A\right)\right) \cup N\left(V\left(G_B\right)\right)$. Here we can show that at most $2b$ vertices burn in any of these configurations. The strategy to do this is for each fire, one firefighter defends the neighbour of the fire which is closest to $G_A$. On the next turn the firefighters each do one of two things. If the firefighter has no clear path to $G_A$ they simply move as far as possible along their path towards $G_A$ and defend that vertex. Otherwise the firefighter has a clear path to $G_A$, so they move through $G_A$, through the path which has no fire, and then through $G_B$ to block any fires moving towards $G_B$. At this point the fire can no longer spread since the initially burning vertices were all of degree two, their neighbours were of degree two, the firefighters defended one of the neighbours, and the firefighters defended the vertex at distance $2$ on the other side from that neighbour. Note that if the fires are close together it could cause this strategy to be invalid, but a combination of removing invalid moves and having the firefighter defend the other side of the fire yields a strategy which will perform just as well in these cases. So in these situations every initially burning vertex burned at most one other vertex. Thus in these initial configurations there are at most $2b \binom{\left(b+1\right)\left(\ell - 2\right)}{b} \in \mathcal{O}\left(\ell^b\right)$ vertices burning.

        We now show that $\binom{|V\left(G_{\ell}\right)|}{b} E_{pr}\left(G_{\ell} - E\left(G_B\right);b,b,\infty\right) \in \Omega\left(\ell^{b+1}\right)$. Select one path (which now terminates at a leaf) to be empty and place one fire on each of the remaining paths so that the fire is at least distance $\frac{\ell}{4}$ from the leaf and from $G_1$. There are $\left(\frac{\ell}{2}\right)^b$ such configurations and in each configuration one of two things happens. Either the firefighters all initially defend on the leaf side of the fires, in which case at least $\frac{\ell}{4}$ vertices burn, or at least one fire has no firefighter between itself and the leaf, in which case at least $\frac{\ell}{4}$ vertices burn. Thus in these configurations at least $\left(\frac{\ell}{4}\right) \left(\frac{\ell}{2}\right)^b \in \Omega\left(\ell^{b+1}\right)$ vertices burn.
    \end{proof}

\begin{figure}[H]
    \centering
\begin{tikzpicture}[scale=0.5]
    \GraphInit[vstyle=Classic]
    \SetUpVertex[FillColor=white]

    \tikzset{VertexStyle/.append style={minimum size=8pt, inner sep=1pt}}

    \Vertex[x=0,y=4,NoLabel=true,]{VA1}
    \Vertex[x=0,y=2,NoLabel=true,]{VA2}
    \Vertex[x=0,y=0,NoLabel=true,]{VA3}
    \Vertex[x=-1,y=2,NoLabel=true,]{VA4}
    
    \Vertex[x=2,y=4,NoLabel=true,]{VI11}
    \Vertex[x=4,y=4,NoLabel=true,]{VI12}
    \Vertex[x=6,y=4,NoLabel=true,]{VI13}
    
    \Vertex[x=2,y=2,NoLabel=true,]{VI21}
    \Vertex[x=4,y=2,NoLabel=true,]{VI22}
    \Vertex[x=6,y=2,NoLabel=true,]{VI23}
     
    \Vertex[x=2,y=0,NoLabel=true,]{VI31}
    \Vertex[x=4,y=0,NoLabel=true,]{VI32}
    \Vertex[x=6,y=0,NoLabel=true,]{VI33}
    
    \Vertex[x=8,y=4,NoLabel=true,]{VB1}
    \Vertex[x=8,y=2,NoLabel=true,]{VB2}
    \Vertex[x=8,y=0,NoLabel=true,]{VB3}

    \Edge(VA4)(VA3)
    \Edge(VA4)(VA2)
    \Edge(VA4)(VA1)

    \Edge(VA1)(VA2)
    \Edge(VA2)(VA3)

    \Edge(VB1)(VB2)
    \Edge(VB2)(VB3)
    
    \Edge(VA1)(VI11)
    \Edge(VI11)(VI12)
    \Edge(VI12)(VI13)
    \Edge(VI13)(VB1)
    
    \Edge(VA2)(VI21)
    \Edge(VI21)(VI22)
    \Edge(VI22)(VI23)
    \Edge(VI23)(VB2)
    
    \Edge(VA3)(VI31)
    \Edge(VI31)(VI32)
    \Edge(VI32)(VI33)
    \Edge(VI33)(VB3)

    \tikzset{EdgeStyle/.append style={bend right=90}}
    \Edge(VA1)(VA3)
    \tikzset{EdgeStyle/.append style={bend left}}
    \Edge(VB1)(VB3)

\end{tikzpicture}
    \caption{$G_3$ with $b=2$ when $G_A$ is $K_4$ and $G_B$ is $K_3$}
    \label{fig:G3}
\end{figure}

    It is also worth noting that we can take $G_A$ to be a tree and we get a spanning subtree of the graph $G_{\ell}$ in Theorem~\ref{thm:nonmono}, which tells us that we can remove as many edges as possible without disconnecting the graph and still have a graph that is worse for the firefighters than the original graph.

\section{Discussion and Open Problems}\label{sec:conclusions}
    
This paper establishes the NP-completeness of {\sc $d$-DPR-$b$-FF} and {\sc $d$-DR-$b$-FF}, but there are still many questions to be asked. For instance, we showed that both problems are NP-complete when $d \geq 2$ but we do not yet know the complexity when $d=1$. Similarly, {\sc $b$-Firefighter} is known to be NP-complete on trees of maximum degree $b+2$ but we do not know which structural restrictions still yield NP-complete instances of {\sc $d$-DPR-$b$-FF} and {\sc $d$-DR-$b$-FF}. We showed in Lemma~\ref{lemma:polytrees} that both problems are tractable on trees under certain mild restrictions but how much structure is required for NP-completeness?

We also discussed the idea of expected damage. Section~\ref{sec:expectedburn} characterizes the optimal graphs with a single fire and firefighter for all values of $d$ with and without path restrictions. We saw however that the lack of monotonicity with path restrictions makes it hard to determine optimal graphs for more fires and firefighters. This lack of monotonicity is particularly interesting and characterizing exactly when a subgraph has smaller expected damage could be an interesting direction of research. From a complexity standpoint one could consider a decision problem where a certain, fixed number $s$ of edges can be added to the graph and the question is whether or not the path-restricted expected damage can be reduced by the addition of at most $s$ edges.

\section*{Acknowledgements}

\noindent Authors A.C. Burgess and D.A. Pike acknowledge research support from NSERC Discovery Grants RGPIN-2019-04328 and RGPIN-2022-03829, respectively. Author J. Marcoux acknowledges support from the Mitacs Accelerate Program and Verafin.

\bibliographystyle{amsplainnodash}
\bibliography{bibliography}

\providecommand{\bysame}{\leavevmode\hbox to3em{\hrulefill}\thinspace}
\providecommand{\MR}{\relax\ifhmode\unskip\space\fi MR }
\providecommand{\MRhref}[2]{%
  \href{http://www.ams.org/mathscinet-getitem?mr=#1}{#2}
}
\providecommand{\href}[2]{#2}
\begin{thebibliography}{10}

\bibitem{arora2009computational}
Sanjeev Arora and Boaz Barak, \emph{Computational complexity: A modern
  approach}, Cambridge University Press, Cambridge, 2009.

\bibitem{barnetson2021firebreak}
Kathleen~D. Barnetson, Andrea~C. Burgess, Jessica Enright, Jared Howell,
  David~A. Pike, and Brady Ryan, \emph{The firebreak problem}, Networks
  \textbf{77} (2021), no.~3, 372--382.

\bibitem{bazgan2013firefighter}
Cristina Bazgan, Morgan Chopin, and Bernard Ries, \emph{The firefighter problem
  with more than one firefighter on trees}, Discrete Appl. Math. \textbf{161}
  (2013), no.~7-8, 899--908.

\bibitem{burgess2022distance}
Andrea~C. Burgess, John Marcoux, and David~A. Pike, \emph{Firefighting with a
  distance-based restriction}, arXiv preprint arXiv:2204.01908 (2022).

\bibitem{cai2009surviving}
Leizhen Cai and Weifan Wang, \emph{The surviving rate of a graph for the
  firefighter problem}, SIAM J. Discrete Math. \textbf{23} (2009/10), no.~4,
  1814--1826.

\bibitem{chen2017continuous}
Xujin Chen, Xiaodong Hu, Changjun Wang, and Ying Zhang, \emph{Continuous
  firefighting on infinite square grids}, Theory and applications of models of
  computation, Lecture Notes in Comput. Sci., vol. 10185, Springer, Cham, 2017,
  pp.~158--171.

\bibitem{crosby2005designing}
Stuart Crosby, Art Finbow, Bert Hartnell, Rania Moussi, Kate Patterson, and
  Dania Wattar, \emph{Designing fire resistant graphs}, vol. 173, 2005, 36th
  Southeastern International Conference on Combinatorics, Graph Theory, and
  Computing, pp.~207--222.

\bibitem{days2019firefighter}
Sarah Days-Merrill, \emph{Firefighter problem played on infinite graphs},
  {H}onors thesis, Bridgewater State University, 2019.

\bibitem{develin2007fire}
Mike Develin and Stephen~G. Hartke, \emph{Fire containment in grids of
  dimension three and higher}, Discrete Appl. Math. \textbf{155} (2007),
  no.~17, 2257--2268.

\bibitem{djikstra1959algorithm}
E.~W. Dijkstra, \emph{A note on two problems in connexion with graphs}, Numer.
  Math. \textbf{1} (1959), 269--271.

\bibitem{edmonds1965paths}
Jack Edmonds, \emph{Paths, trees, and flowers}, Canadian J. Math. \textbf{17}
  (1965), 449--467.

\bibitem{finbow2000minimizing}
Stephen Finbow, Bert Hartnell, Qiyan Li, and Kyle Schmeisser, \emph{On
  minimizing the effects of fire or a virus on a network}, J. Combin. Math.
  Combin. Comput. \textbf{33} (2000), 311--322.

\bibitem{finbow2007npcomplete}
Stephen Finbow, Andrew King, Gary MacGillivray, and Romeo Rizzi, \emph{The
  firefighter problem for graphs of maximum degree three}, Discrete Math.
  \textbf{307} (2007), no.~16, 2094--2105.

\bibitem{finbow2009firefighter}
Stephen Finbow and Gary MacGillivray, \emph{The firefighter problem: a survey
  of results, directions and questions}, Australas. J. Combin. \textbf{43}
  (2009), 57--77.

\bibitem{greenlaw1995parallel}
Raymond Greenlaw, H.~James Hoover, and Walter~L. Ruzzo, \emph{Limits to
  parallel computation: {${\rm P}$}-completeness theory}, The Clarendon Press,
  Oxford University Press, New York, 1995.

\bibitem{hartke2006attempting}
Stephen~G. Hartke, \emph{Attempting to narrow the integrality gap for the
  firefighter problem on trees}, Discrete methods in epidemiology, DIMACS Ser.
  Discrete Math. Theoret. Comput. Sci., vol.~70, Amer. Math. Soc., Providence,
  RI, 2006, pp.~225--231.

\bibitem{hartnell1995firefighter}
Bert Hartnell, \emph{Firefighter! {A}n application of domination}, 25th
  Manitoba Conference on Combinatorial Mathematics and Computing, 1995.

\bibitem{huangfu2018parallelizing}
Qi~Huangfu and J.~A.~Julian Hall, \emph{Parallelizing the dual revised simplex
  method}, Math. Program. Comput. \textbf{10} (2018), no.~1, 119--142.

\bibitem{karp1972reducibility}
Richard~M. Karp, \emph{Reducibility among combinatorial problems}, Complexity
  of computer computations ({P}roc. {S}ympos., {IBM} {T}homas {J}. {W}atson
  {R}es. {C}enter, {Y}orktown {H}eights, {N}.{Y}., 1972), The IBM Research
  Symposia Series, Plenum, New York-London, 1972, pp.~85--103.

\bibitem{reingold2008logspace}
Omer Reingold, \emph{Undirected connectivity in log-space}, J. ACM \textbf{55}
  (2008), no.~4, Art. 17, 24.

\bibitem{wagner2021new}
Connor Wagner, \emph{A new survey on the firefighter problem}, Masters thesis,
  University of Victoria, 2021.

\bibitem{west2000introduction}
Douglas~B. West, \emph{Introduction to graph theory}, 2 ed., Prentice Hall,
  September 2000.

\end{thebibliography}

\end{document}